\documentclass[12pt]{article}
\usepackage[english]{babel}
\usepackage{amsfonts,amsmath,amsxtra,amsthm,amssymb,latexsym}
\usepackage{hyperref}
\usepackage[usenames,dvipsnames]{color}
\usepackage{verbatim}

\definecolor{light-gray}{gray}{0.50}
\newtheorem{theorem}{Theorem}[section]

\newtheorem{definition}[theorem]{Definition}
\newtheorem{remark}[theorem]{Remark}
\newtheorem{example}[theorem]{Example}
\newtheorem{proposition}[theorem]{Proposition}

\numberwithin{equation}{section}

 \DeclareMathOperator{\dom}{dom}
 \DeclareMathOperator{\Ext}{Ext}

\DeclareMathOperator{\loc}{loc}

\newcommand\gH{{\mathfrak{H}}}
\newcommand{\gG}{{\Gamma}}

\newcommand\cH{{\mathcal{H}}}

\newcommand\cN{{\mathfrak{N}}}

\newcommand\I{{\rm{i}}}

\def\Ext{{\rm Ext}}

\def\wt#1{{{\widetilde #1} }}

\newcommand{\gotH}{{\mathfrak H}}

\newcommand{\adots}{\makebox[0.9ex][l]{\raisebox{-0.2ex}{.}} \raisebox{0.4ex}{.}
\makebox[0.9ex][r]{\raisebox{1.1ex}{.}}}

\title{\bf Krein extension of a differential operator of even order}

\author{\bf Yaroslav I. Granovskyi, Leonid L. Oridoroga}
\date{}

\begin{document}


\maketitle


{\bf Abstract.}
We describe the Krein extension of minimal operator associated with the expression $\mathcal{A}:=(-1)^n\frac{d^{2n}}{dx^{2n}}$ on a finite interval $(a,b)$ in terms of boundary conditions. All non-negative extensions of the operator $A$ as well as extensions with a finite number of negative squares are described.


\section{Introduction}

Let $A$ be a semi-bounded symmetric operator in a separable Hilbert space $\mathfrak{H}.$
It is well known that the operator $A$ admits self-adjoint extensions preserving the lower bound
(see~\cite[Ch. VIII]{AkhGl} and~\cite[I]{Kre47}).
According to the classical Krein's result ~\cite[I]{Kre47}, in the set $\Ext_A(0,\infty)$ of all non-negative self-adjoint extensions of the operator $A$, there exist two "extreme" extensions $\widehat{A}_F$ and $\widehat{A}_K$ uniquely determined by the following inequalities:
\begin{equation}\label{muM}
\left(\widehat{A}_F+x\right)^{-1}\leq\left(\widetilde{A}+x\right)^{-1}
\leq\left(\widehat{A}_K+x\right)^{-1}, \quad x\in(0, \infty), \quad \widetilde{A}\in\Ext_{A}(0,\infty).
\end{equation}
The extension $\widehat{A}_F$ is called Friedrichs' (or a hard), and the extension $\widehat{A}_K$ is called Krein's (or a soft), see~\cite[I]{Kre47}. In the case of positively definite operator $A>\varepsilon I>0$,
M.G. Krein showed that
\begin{equation}\label{Kr_dom}
\widehat{A}_K=A^*\upharpoonright\left(\dom A\dotplus\ker A^*\right).
\end{equation}
(see \cite[I]{Kre47}).

In the case of non-negative operator $A\geq 0$, the extensions $\widehat{A}_F$ and $\widehat{A}_K$ were first described in~\cite{DerMal91} in terms of abstract boundary conditions. Namely, it was shown that
\begin{equation}\label{DerMal}
\dom \widehat{A}_K=\{f\in\dom A^*:\Gamma_1f=M(0)\Gamma_0f\}, \quad \dom \widehat{A}_F=\{f\in\dom A^*:\Gamma_1f=M(-\infty)\Gamma_0f\},
\end{equation}
where $M(0)=M(0-)$ is a limit value of the Weyl function at zero, and $M(-\infty)$ is a limit value of the Weyl function at infinity (see Definition~\ref{def_Weylfunc}).

Description of the Friedrichs extension independent of~\eqref{DerMal} is known in many cases.
For instance, M.G. Krein showed that for ordinary differential operators on a finite interval
extension $\widehat{A}_F$ is generated by the Dirichlet problem~(\cite[II]{Kre47}).

H. Kalf in \cite{Kalf} investigated the general three-term Sturm-Liouville differential expression
\begin{equation}\label{Kalf-1}
\tau u=\frac{1}{k}\left[-(pu')'+qu\right]
\end{equation}
on an interval $(0, \infty)$ under the following assumptions on coefficients:
\begin{itemize}
\item[\textit{(i)}]
$k,p>0$ a.e. on $(0, \infty);$ \quad $k, 1/p\in L^1_{\loc}(0, \infty);$ \quad $q\in L^1_{\loc}(0, \infty)$ is real-valued;
\item[\textit{(ii)}]
There exists a number $\mu\in\mathbb{R}$ and functions $g_0,g_{\infty}\in AC_{\loc}(0, \infty)$ with $pg_0',pg_{\infty}'\in AC_{\loc}(0, \infty)$ and
$g_0>0$ near $0$, $g_{\infty}>0$ near $\infty$ such that
\begin{equation}\label{Kalf-2}
\int_0 \frac{1}{pg_0^2}=\int^{\infty}\frac{1}{pg_{\infty}^2}=\infty
\end{equation}
and
\begin{equation}\label{Kalf-3}
q\geqslant \frac{(pg_0')'}{g_0}-\mu k\quad \text{near} \quad 0, \quad q\geqslant \frac{(pg_{\infty}')'}{g_{\infty}}-\mu k\quad \text{near} \quad \infty.
\end{equation}
\end{itemize}
The main result of the paper \cite{Kalf} is the following description of Friedrichs' extension $\widehat{T}_F$ of the minimal operator $T_{\min}$ associated with \eqref{Kalf-1}:
\begin{equation}\label{Kalf-4}
\dom\widehat{T}_F=\left\{u\in\dom T_{\max}: \quad\int_0 pg_0^2\left|\left(\frac{u}{g_0}\right)'\right|^2<\infty, \quad\int^{\infty} pg_{\infty}^2\left|\left(\frac{u}{g_{\infty}}\right)'\right|^2<\infty\right\}.
\end{equation}
For more information see \cite[Theorem 1]{Kalf} and related remarks.


%

This result has been extended in \cite{Ges-3} to the case of singular differential operators on arbitrary intervals $(a,b) \subseteq \mathbb{R}$ associated with four-term general differential expressions of the type
\begin{equation}\label{17-1}
\tau u = \frac{1}{k} \left( - \big(p[u' + s u]\big)' + s p[u' + s u] + qu\right):=\frac{1}{k} \left(-(u^{[1]})'+su^{[1]}+qu\right),
\end{equation}
where the coefficients $p$, $q$, $k$, $s,$ are real-valued and Lebesgue measurable on $(a,b)$, with $p\neq 0$, $k>0$ a.e.\ on $(a,b)$, and $p^{-1}$, $q$, $k$, $s \in L^1_{\text{loc}}((a,b); dx)$, and $u$ is supposed to satisfy
\begin{equation}\label{17-2}
u \in AC_{\text{loc}}(a,b), \quad u^{[1]} \in AC_{\text{loc}}(a,b).
\end{equation}
In particular, this setup implies that $\tau$ permits a distributional potential coefficient,
including potentials in $H^{-1}_{\text{loc}}(a,b).$

Imposing additional to \eqref{Kalf-2}--\eqref{Kalf-3} assumptions on coefficients, the authors characterize the Friedrichs extension of $T_{\text{min}}$ by the same conditions \eqref{Kalf-4}.
For more details see \cite[Theorems 11.17 and 11.19]{Ges-3}.

In \cite{Ges-3} it is also described the Krein extension of $T_{\text{min}}$ on a finite interval $(a,b)$ in the special case where $\tau$ is regular (i.e. $p^{-1}, q, k$ and $s$ are integrable near $a$ and $b$). A description is given as follows:
\begin{equation}\label{Gesztesy-16-2}
\dom \widehat{T}_K=\left\{g\in\dom T_{\max}: \quad
\begin{pmatrix}
g(b)\\
g^{[1]}(b)
\end{pmatrix}
=R_K
\begin{pmatrix}
g(a)\\
g^{[1]}(a)
\end{pmatrix}
\right\},
\end{equation}
where
\begin{equation}\label{R_K}
R_K=\frac{1}{u_1^{[1]}(a)}
\begin{pmatrix}
-u_2^{[1]}(a) & 1\\
u_1^{[1]}(a)u_2^{[1]}(b)-u_1^{[1]}(b)u_2^{[1]}(a) & u_1^{[1]}(b)
\end{pmatrix},
\end{equation}
and $u_j(\cdot), j\in\{1,2\},$ are positive solutions of $\tau u=0$ determined by the conditions
\begin{equation}
\begin{gathered}
u_1(a)=0, \quad u_1(b)=1,\\
u_2(a)=1, \quad u_2(b)=0.
\end{gathered}
\end{equation}
For more details see \cite[Theorem 12.3]{Ges-3}.

Several papers (see \cite{Ges-3}--\cite{Ful-2}, \cite{Kalf}--\cite{Kostes} and the references therein) are devoted to the spectral analysis of boundary value problems for the one-parametric Bessel's differential expression
\begin{equation}\label{Bes}
\tau_{\nu}=-\frac{d^2}{dx^2}+\frac{\nu^2-\frac{1}{4}}{x^2}, \quad \nu\in[0,1)\setminus\left\{1/2\right\}.
\end{equation}
We especially note the papers of H. Kalf and W. Everitt \cite{Kalf, EvKalf}, where the explicit form of the Weyl-Titchmarsh $m-$coefficient
of the expression $\tau_{\nu}$ in $L^2(\mathbb{R}_+)$ was found.

In \cite{AnBud, Bru, EvKalf, Kalf}, there were described domains of the Friedrichs extension for the minimal operator $A_{\nu, \infty}$ associated with expression \eqref{Bes} in $L^2(\mathbb{R}_+)$. In \cite{EvKalf} the same was done for all self-adjoint extensions of the operator   $A_{\nu, \infty}$. The most complete result was obtained in \cite{AnBud}.
Namely, $\widehat{A}_{\nu,\infty,F}$ and $\widehat{A}_{\nu,\infty,K}$ are the restrictions of the maximal operator $A_{\nu,\infty}^*=A_{\nu,\infty, \max}$ to the domains
\begin{equation}
\dom \widehat{A}_{\nu,\infty,F}=\left\{f\in \dom{A^*_{\nu,\infty}}
:[f, x^{\frac{1}{2}+\nu}]_0=0\right\}
\end{equation}
and
\begin{equation}
\dom\widehat{A}_{\nu,\infty,K}=\left\{
\begin{array}{ll}
\{f\in \dom{A^*_{\nu,\infty}}:[f, x^{\frac{1}{2}-\nu}]_0=0\}, &
\nu\in(0,1),\\
\{f\in \dom{A^*_{0,\infty}} :[f, x^{\frac{1}{2}}]_0=0\},
&\nu=0
\end{array}
\right.
\end{equation}
respectively, where
\begin{equation}\label{A_max}
\dom{A^*_{\nu,\infty}}= \left\{
\begin{array}{ll}
{H}_0^2(\mathbb R_+)\dot{+}\text{span}\{x^{1/2+\nu}\xi(x),
x^{1/2-\nu}\xi(x)\}, \nu\in(0,1),\\
{H}_0^2(\mathbb R_+)\dotplus\text{span}\{x^{1/2}\xi(x), x^{1/2}\log(x)\xi(x)\}, \nu=0.
\end{array}
\right.\\
\end{equation}
Here $[f,g]_x:=f(x)\overline{g'(x)}-f'(x)\overline{g(x)}$ for all $x\in\mathbb{R}_+,$ and $\xi\in C_0^2(\mathbb R_+)$ is a function such that $\xi(x)=1$
{whenever} $x\in[0,1].$ For more details see \cite[Proposition 5.7 and Remark 5.8]{AnBud}.

Friedrichs' and Krein's extensions $\widehat{A}_{\nu,b,F}$ and $\widehat{A}_{\nu,b,K}$ of the minimal operator corresponding to \eqref{Bes} on a finite interval $(0,b)$ were also described there (see \cite[Proposition 4.5]{AnBud}).

M.G. Krein (\cite[II]{Kre47}) investigated certain extensions of the minimal operator $T_{\min}$ associated in $L^2(a,b)$ with the following quasi-derivative expression
\begin{equation}\label{Krei-1}
Tf:=f^{[2n]},
\end{equation}
where
\begin{equation}\label{Krei-2}
f^{[2n]}:=p_n f-\frac{d}{dx}\left[p_{n-1}f^{(1)}-\frac{d}{dx}[p_{n-2}f^{(2)}-...-\frac{d}{dx}[p_1f^{(n-1)}-\frac{d}{dx}(p_0f^{(n)})]...]   \right].
\end{equation}
In the case of sufficiently smooth coefficients $p_k, k\in\{0,1,...,n\},$ expression \eqref{Krei-2} can be written in the Jacobi-Bertrand form:
\begin{equation}\label{J-B}
f^{[2n]}=\sum_{k=0}^n (-1)^k \frac{d^k}{dx^k}\left(p_{n-k}\frac{d^kf}{dx^k}\right).
\end{equation}
In \cite[II]{Kre47}  Friedrichs' extension of the minimal operator $T_{\min}$ corresponds to Dirichlet realization:
\begin{equation}\label{Krei-3}
\dom \widehat{T}_F = \{f\in \dom T_{\max}: f^{[k]}(a)=f^{[k]}(b)=0, \quad k\in\{0,1,...,n-1\}\}.
\end{equation}

In the paper by A. Lunyov \cite{Lun} the spectral properties of the operator $A$ generated in $L^{2}(\mathbb{R}_+)$ by the differential expression
\begin{equation}\label{Lun-1}
l:=(-1)^{n}\frac{d^{2n}}{dx^{2n}}
\end{equation}
are investigated, and the Krein extension of the corresponding minimal operator $A_{\min}$ in terms of boundary conditions is described in the following way:
\begin{equation}\label{Lun-K}
y^{(n)}(0)=y^{(n+1)}(0)=...=y^{(2n-1)}(0)=0.
\end{equation}

Using the technique of boundary triplets and the corresponding Weyl functions the author found explicit form of the characteristic matrix and the corresponding spectral function for the Friedrichs and Krein extensions of the minimal operator $A_{\min}$ (see \cite[Theorems 1 and 2]{Lun}).

In \cite{Ges-2} the unitary equivalence of the inverse of the Krein extension
(on the orthogonal complement of its kernel) of a densely defined, closed, strictly positive
operator, $S\geq \varepsilon I_{\mathcal{H}}$ for some $\varepsilon >0$ in a Hilbert space $\mathcal{H}$ to an abstract buckling problem operator is proved.
%

Several papers are devoted to Friedrichs' and Krein's extensions of perturbed Laplacian on bounded and unbounded domains.

For instance on the subject of semibounded extensions of non-negative symmetric operators we refer to M.Sh. Birman \cite{Birman}, G. Grubb \cite{Grubb-1} (elliptic operators on bounded domains with smooth boundary), J. Behrndt et al. \cite{Behrndt-1, Behrndt-2} (elliptic operators on Lipschitz domains), F. Gesztesy and M. Mitrea \cite{Ges-4} (Laplacian on non-smooth domains).

In \cite{Ges-1} the authors study spectral properties for $\widehat{H}_{K,\Omega}$, the Krein
extension of the perturbed Laplacian $-\Delta+V$ defined on
$C^\infty_0(\Omega)$, where $V$ is measurable, bounded and nonnegative, in
a bounded open set $\Omega\subset\mathbb{R}^n$ belonging to a class of
nonsmooth domains which contains all convex domains, along with all domains
of class $C^{1,r}$, $r>1/2.$ 

See also ~\cite{AnNis, Arl, AGMST, Bro, Ges-5, Grubb, Hart, HasMal, Nen} and the references therein.

However, the problem of finding $M(0)$ is nontrivial even in the case of positively definite operator. Its solution is known in some cases --- see papers \cite[Theorem 1.1]{Bru}, \cite[Proposition 4.5 (ii), Proposition 5.7 (ii)]{AnBud}, \cite[Theorem 1]{Bro}, \cite[Theorem 2]{Lun} mentioned above.







Here we consider the minimal operator $A:=A_{\min}$ associated with the differential expression
\begin{equation}\label{Op. A}
\mathcal{A}:=(-1)^n\frac{d^{2n}}{dx^{2n}}
\end{equation}
on a finite interval $(a,b),$ we describe its Krein's extension in terms of boundary conditions.
In this way we find $M(0)$ for special (natural) boundary triplet for $A^*$. Note that the corresponding boundary operator is expressed by means of blocks of certain auxiliary Toeplitz matrix (see~\eqref{B_K}). Using the technique of boundary triplets and the corresponding Weyl functions developed in~\cite{DerMal91} we describe all non-negative extensions of $A_{\min}$ as well as extensions with the finite negative spectrum.

\section{Preliminaries}
Let  $A$ be a densely defined closed symmetric operator in a separable
Hilbert space $\gH$ with equal deficiency indices $\mathrm{n}_\pm(A)=\dim(\cN_{\pm \I})
\leq \infty,$ where $\cN_z:=\ker(A^*-z)$ is the defect subspace.


\begin{definition}
A closed extension $A'$ of $A$ is called a \emph{proper
one} if $A\subset A' \subset A^*$.  The set of all proper extensions of  $A$
completed by the (non-proper) extensions $A$ and $A^*$ is  denoted  by $\Ext_A$.
\end{definition}

Assume that operator $A\in\mathcal{C}(\mathfrak{H})$ is non-negative.
Then the set $\Ext_{A}(0,\infty)$ of its non-negative self-adjoint extensions is
non-empty (see~\cite{AkhGl, Kato66}). Moreover, there is a maximal non-negative extension $\widehat{A}_F$ (also called Friedrichs' or hard extension), and there is a minimal non-negative extension $\widehat{A}_K$ (Krein's or soft extension) satisfying~\eqref{muM}. For details we refer the reader to~\cite{AkhGl, GorGor91}.

\begin{definition}[\cite{GorGor91}]\label{def_ordinary_bt}
A triplet $\Pi=\{\cH,\gG_0,\gG_1\}$ is called a {\rm boundary triplet} for
the adjoint operator $A^*$ if $\cH$ is an auxiliary Hilbert space and
$\Gamma_0,\Gamma_1:\  \dom A^*\rightarrow \cH$ are linear mappings such that the
abstract Green identity
\begin{equation}\label{II.1.2_green_f}
(A^*f,g)_\gH - (f,A^*g)_\gH = (\gG_1f,\gG_0g)_\cH - (\gG_0f,\gG_1g)_\cH, \qquad
f,g\in\dom A^*,
\end{equation}
holds and the mapping $\gG:=\begin{pmatrix}\Gamma_0\\\Gamma_1\end{pmatrix}:  \dom A^*
\rightarrow \cH \oplus \cH$ is surjective.
\end{definition}

First,  note that a boundary triplet for $A^*$ exists whenever  the deficiency indices
of $A$ are equal, $\mathrm{n}_+(A)= \mathrm{n}_-(A)$.  Moreover, $\mathrm{n}_\pm(A) =
\dim\cH$ and $\ker\Gamma = \ker\Gamma_0 \cap \ker\Gamma_1= \dom A$. Note also
that $\Gamma$ is a bounded mapping from $\gotH_+ = \dom A^*$ equipped with the graph
norm to $\cH\oplus\cH.$

A boundary triplet for $A^*$ is not unique. Moreover, for any self-adjoint extension
$\wt A := \wt A^*$ of $A$ there exists a boundary triplet $\Pi=\{\cH,\gG_0,\gG_1\}$ such
that $\ker\Gamma_0 = \dom \wt A$.



\begin{definition}[{\cite{DerMal91}}]\label{def_Weylfunc}
Let $A$ be a densely defined closed symmetric operator in $\gH$ with equal deficiency
indices, and let $\Pi=\{\cH,\gG_0,\gG_1\}$ be a boundary triplet for $A^*$. The operator
valued functions $\gamma(\cdot) :\rho(A_0)\rightarrow  \mathcal{B}(\cH,\gH)$ and
$M(\cdot):\rho(A_0)\rightarrow  \mathcal{B}(\cH), A_0:=A^*\!\upharpoonright\ker\gG_0$
defined by
\begin{equation}\label{II.1.3_01}
\gamma(z):=\bigl(\Gamma_0\!\upharpoonright\cN_z\bigr)^{-1} \qquad\text{and}\qquad
M(z):=\Gamma_1\gamma(z), \qquad z\in\rho(A_0),
\end{equation}
are called the {\em $\gamma$-field} and the {\em Weyl function}, respectively,
corresponding to the boundary triplet $\Pi.$
\end{definition}

\begin{remark}[{\cite{R-B}, \cite[Ch. VIII]{AkhGl}}]\label{A_CD}
 In the case of $\mathrm{n}_{\pm}(A)=m<\infty$, the set of all self-adjoint extensions of the operator $A$ is parametrized as follows:
\begin{equation}\label{C-D}
\begin{gathered}
\Ext_A\ni\widetilde{A}=\widetilde{A}^*=A_{C,D}=A^*\upharpoonright\ker(D\Gamma_1-C\Gamma_0), \\
where \quad CD^*=DC^*,\quad \det(CC^*+DD^*)\neq 0, \quad C, D\in\mathbb{C}^{m\times m}.
\end{gathered}
\end{equation}
\end{remark}

%
%
%
%

\begin{definition}\label{kappa}
Let $T$ be a self-adjoint operator in $\mathfrak{H}$, and let $E_T(\cdot)$ be its spectral measure.
It is said that the operator $T$ has $\kappa$ negative eigenvalues if
\begin{equation}\label{kappa_numbers}
\kappa_-(T):=\dim E_T(-\infty,0)=\kappa.
\end{equation}
\end{definition}

In the following proposition all self-adjoint extensions of an operator $A\geq 0$ with a finite negative spectrum are described.

\begin{proposition}[{\cite{DerMal91, DerMal95}}]\label{A_0}
Let $A$ be  a densely defined non-negative symmetric operator  in $\gotH$, $\mathrm{n}_{\pm}(A)=m<\infty,$ let $\Pi=\{\cH,\gG_0,\gG_1\}$ be a boundary triplet for  $A^*$ such that  $A_0\geq 0,$ and let $A_{C,D}$ be an arbitrary self-adjoint extension of the form~\eqref{C-D}.
Let  also  $M(\cdot)$ be the corresponding Weyl function.  Then:
\begin{itemize}
\item[\textit{(i)}]  There exist  strong   resolvent  limits
\begin{equation}\label{W_lim}
M(0):=s-R-\lim\limits_{x\uparrow 0}M(x), \qquad M(-\infty):=s-R-\lim\limits_{x\downarrow
-\infty}M(x).
\end{equation}
%
 \item[\textit{(ii)}] $\dom A_0\cap\dom \widehat{A}_K=\dom A$ \quad $(\dom A_0\cap\dom \widehat{A}_F=\dom A)$
 if and only if $M(0)\in\mathbb{C}^{m\times m}$ $(M(-\infty)\in\mathbb{C}^{m\times m})$. Moreover, in this case
\begin{equation}\label{A_K_M(0)}
\widehat{A}_K=A^*\upharpoonright{\ker}\left(\Gamma_1-M(0)\Gamma_0\right), \quad \left(\widehat{A}_F=A^*\upharpoonright{\ker}\left(\Gamma_1-M(-\infty)\Gamma_0\right)\right).
\end{equation}
 \item[\textit{(iii)}] $A_0=\widehat{A}_F$
($A_0=\widehat{A}_K$) if and only if
\begin{equation}\label{WAFK}
\lim_{x\downarrow -\infty}\left(M(x)f,f \right)=-\infty \quad
\left(\lim_{x\uparrow 0}\left(M(x)f,f \right)=+\infty\right), \quad f\in\mathcal{H}\setminus\{0\}.
\end{equation}

\item[\textit{(iv)}] If $A_0=\widehat{A}_F,$ then the following identity holds:
\begin{equation}\label{kappa}
 \kappa_-(A_{C,D})=\kappa_-(CD^*-DM(0)D^*).
\end{equation}
In particular, $A_{C,D}\geq 0$ if and only if $CD^*-DM(0)D^*\geq 0.$
 \item[\textit{(v)}] The extension $A_B=A^*\upharpoonright\ker(\Gamma_1-B\Gamma_0)$ is symmetric (self-adjoint) if and only if $B$ is symmetric (self-adjoint).
  \end{itemize}
\end{proposition}

\begin{theorem}[{\cite[I, Theorem 14]{Kre47}}]\label{Kr-2}
Let $A$ be a symmetric positively definite operator. Then
$\dom\widehat{A}_K=\dom A\dotplus\cN_0$, and
\begin{equation}\label{eq: T}
\widehat{A}_K(f+f_0)=Af\quad\mbox{for any}\quad f\in\dom A,\,\, f_0\in\cN_0.
\end{equation}
\end{theorem}

\section{Main result}
Let $A:=A_{\min}$ be the minimal operator generated in $\mathfrak{H}=L^2\left(a,b\right), -\infty<a<b<\infty$ by the differential expression \eqref{Op. A}. In view of~\cite{DerMal92},
the boundary triplet for $A^*:=A_{\max}$ can be taken as
\begin{equation}\label{bt}
\mathcal{H}=\mathbb{C}^{2n}, \quad
\Gamma_0 f=
\begin{pmatrix}
f(a) \\
\vdots \\
f^{(n-1)}(a) \\
f(b) \\
\vdots \\
f^{(n-1)}(b)
\end{pmatrix}, \quad \Gamma_1 f=
\begin{pmatrix}
(-1)^{n-1}f^{(2n-1)}(a) \\
\vdots \\
f^{(n)}(a) \\
(-1)^{n}f^{(2n-1)}(b) \\
\vdots \\
-f^{(n)}(b)
\end{pmatrix}.
\end{equation}

The main result of this paper is presented by the following theorem.

\begin{theorem}\label{A_K} Let $A$ be the minimal operator defined by~\eqref{Op. A}.
Let also $\Pi=\left\{\mathcal{H}, \Gamma_0, \Gamma_1\right\}$ be the boundary triplet for $A^*$ defined by relations~\eqref{bt}. Then the following assertions hold.
\item[(i)]
The domain of Krein's extension $\widehat{A}_K$ is of the form
\begin{equation}\label{dom_A_K}
\dom\widehat{A}_K=\left\{f\in W^{2n,2}(a,b):
\begin{pmatrix}
f^{(2n-1)}(b) \\
\vdots \\
f(b)
\end{pmatrix}
=T\begin{pmatrix}
f^{(2n-1)}(a) \\
\vdots \\
f(a)
\end{pmatrix}\right\},
\end{equation}
where $T$ is the Toeplitz lower-triangular $2n\times 2n$ matrix of the form
\begin{equation}\label{Matr_T}
T=\begin{pmatrix}
1 & \dots & & & {\bf 0} \\
b-a & 1 & \dots \\
\hdotsfor{5}   \\
\frac{(b-a)^{2n-1}}{(2n-1)!} & \frac{(b-a)^{2n-2}}{(2n-2)!} & \dots & b-a & 1
\end{pmatrix}.
\end{equation}
\item[(ii)]
Krein's extension $\widehat{A}_K$ is given by
\begin{equation}\label{domA_K-0}
\dom\widehat{A}_K=\left\{f\in W^{2n,2}(a,b):\Gamma_1f=B_K\Gamma_0f\right\},
\end{equation}
where
\begin{equation}\label{B_K}
B_K=\begin{pmatrix}
QT_2^{-1}T_1S & -QT_2^{-1}S \\
-QT_1T_2^{-1}T_1S & QT_1T_2^{-1}S
\end{pmatrix},
\end{equation}
and $T_1, T_2, Q, S$ are the following $n\times n$ matrices:
\begin{equation}\label{TCS}
\begin{gathered}
T_1=\begin{pmatrix}
1 & \dots & & & {\bf 0} \\
b-a & 1 & \dots \\
\hdotsfor{5}   \\
\frac{(b-a)^{n-1}}{(n-1)!} & \frac{(b-a)^{n-2}}{(n-2)!} & \dots & b-a & 1
\end{pmatrix}, \quad T_2=\begin{pmatrix}
\frac{(b-a)^n}{n!} & \frac{(b-a)^{n-1}}{(n-1)!} & \dots & b-a \\
\frac{(b-a)^{n+1}}{(n+1)!} & \frac{(b-a)^n}{n!} & \dots & \frac{(b-a)^2}{2!} \\
\hdotsfor{4}   \\
\frac{(b-a)^{2n-1}}{(2n-1)!} & \frac{(b-a)^{2n-2}}{(2n-2)!} & \dots & \frac{(b-a)^n}{n!}
\end{pmatrix},\\
Q=\begin{pmatrix}
(-1)^n & & {\bf 0} \\
& \ddots \\
{\bf 0} & & -1
\end{pmatrix}, \quad S=\begin{pmatrix}
  {\bf 0} &  & 1 \\
          & \adots &  \\
    1 & & {\bf 0}
\end{pmatrix}.
\end{gathered}
\end{equation}
\end{theorem}
\begin{proof}
(i) Let us consider the $k-$th row in \eqref{dom_A_K}:
\begin{equation}\label{k_th_row}
f^{(2n-k)}(b)=\sum_{m=1}^k\frac{f^{(2n-m)}(a)}{(2n-m)!}(b-a)^{2n-m}, \quad k\in\left\{1,2,\dots,n\right\}.
\end{equation}
Due to the Theorem \ref{Kr-2}, it suffices to prove~\eqref{k_th_row} for
$\ker A^*=\text{span}\left\{1,x,...,x^{2n-1}\right\}.$ Since $\ker A^*$
consists of polynomials of degree not greater than $2n-1$, the formula \eqref{k_th_row}
follows from Tailor's one for polynomials.

(ii) Let
\begin{equation}\label{U_j}
\begin{gathered}
U_1=\begin{pmatrix}
f^{(n-1)}(b) \\
\vdots \\
f(b)
\end{pmatrix}, \quad U_2=\begin{pmatrix}
f^{(n-1)}(a) \\
\vdots \\
f(a)
\end{pmatrix}, \quad U_3=\begin{pmatrix}
f^{(2n-1)}(b) \\
\vdots \\
f^{(n)}(b)
\end{pmatrix}, \quad U_4=\begin{pmatrix}
f^{(2n-1)}(a) \\
\vdots \\
f^{(n)}(a)
\end{pmatrix}, \\
U_{1,t}=\begin{pmatrix}
f(b) \\
\vdots \\
f^{(n-1)}(b)
\end{pmatrix}, \quad U_{2,t}=\begin{pmatrix}
f(a) \\
\vdots \\
f^{(n-1)}(a)
\end{pmatrix}.
\end{gathered}
\end{equation}
Then
\begin{equation}\label{Gamma}
\Gamma_0 f=\begin{pmatrix}
SU_2 \\
SU_1
\end{pmatrix}=\begin{pmatrix}
U_{2,t} \\
U_{1,t}
\end{pmatrix}, \quad \Gamma_1 f=\begin{pmatrix}
-QU_4 \\
QU_3
\end{pmatrix},
\end{equation}
and hence the equality in \eqref{dom_A_K} takes the form
\begin{equation}\label{dom_A_K-1}
\begin{pmatrix}
U_3 \\
U_1
\end{pmatrix}=\begin{pmatrix}
T_1 & \mathbb{O} \\
T_2 & T_1
\end{pmatrix}
\begin{pmatrix}
U_4 \\
U_2
\end{pmatrix} \quad \text{or} \quad\begin{cases}
U_3=T_1U_4+\mathbb{O}U_2 \\
U_1=T_2U_4+T_1U_2
\end{cases}.
\end{equation}
Expressing $U_4$ and $U_3$ from the latter we get
\begin{equation}\label{U_4_U_3}
\begin{cases}
U_4=T_2^{-1}U_1-T_2^{-1}T_1U_2 \\
U_3=T_1T_2^{-1}U_1-T_1T_2^{-1}T_1U_2
\end{cases}.
\end{equation}
Multiplying from the left the first equality by $-Q$ and the second one by $Q$  we obtain
\begin{equation}\label{C}
\begin{cases}
-QU_4=-QT_2^{-1}U_1+QT_2^{-1}T_1U_2 \\
QU_3=QT_1T_2^{-1}U_1-QT_1T_2^{-1}T_1U_2
\end{cases}.
\end{equation}
Since $U_1=SU_{1,t}, \quad U_2=SU_{2,t}$ then \eqref{C} yields
\begin{equation}\label{CU_t}
\begin{cases}
-QU_4=QT_2^{-1}T_1SU_{2,t}-QT_2^{-1}SU_{1,t}\\
QU_3=-QT_1T_2^{-1}T_1SU_{2,t}+QT_1T_2^{-1}SU_{1,t}
\end{cases}
\end{equation}
or
\begin{equation}\label{B_K-1}
\Gamma_1 f =\begin{pmatrix}
QT_2^{-1}T_1S & -QT_2^{-1}S \\
-QT_1T_2^{-1}T_1S & QT_1T_2^{-1}S
\end{pmatrix}\Gamma_0 f.
\end{equation}
Thus, we arrive at the representation $\Gamma_1 f = B_K\Gamma_0 f,$ and the equality \eqref{B_K} is proved.
\end{proof}

\begin{theorem}\label{s_a_B_K}
The matrix $B_K$ is self-adjoint, i.e., $B_K=B^*_K.$
\end{theorem}

\begin{proof}
Obviously, $B_K$ is self-adjoint in accordance with Proposition~\ref{A_0} (v). Let us prove
this fact directly. It is necessary to show that the following equalities hold:
\begin{align}
QT_2^{-1}T_1S=\left(QT_2^{-1}T_1S\right)^*,\label{eq1}\\
QT_1T_2^{-1}S=\left(QT_1T_2^{-1}S\right)^*,\label{eq2}\\
QT_2^{-1}S=\left(QT_1T_2^{-1}T_1S\right)^*.\label{eq3}
\end{align}
Denote $V=ST_2.$ Let us prove the equality~\eqref{eq1}. We start with the following
obvious relation:
\begin{equation}\label{eq1a}
QT_2^{-1}T_1S=QT_2^{-1}SST_1S=QV^{-1}T_1^*.
\end{equation}
Let us check that inverse matrix $T_1^{-1*}VQ$ is self-adjoint.

We will numerate matrix entries of $V$ starting from its right low corner ($j$ is the number of a column and $k$ is the number of a row):
$v_{j,k}=\frac{(b-a)^{j+k-1}}{(j+k-1)!}.$

The entry $T_1^{-1*}V$ (denoted by $\varphi_{j,k}$) has the following form:
\begin{equation}\label{phi}
\varphi_{j,k}=\sum_{l=0}^{k-1}\frac{(a-b)^{l}}{l!}v_{j,k-l}=(b-a)^{j+k-1}\sum_{l=0}^{k-1}(-1)^{l}\frac{1}{l!(j+k-l-1)!}.
\end{equation}
The symmetric one is
\begin{equation}\label{phi-2}
\varphi_{k,j}=(b-a)^{j+k-1}\sum_{m=0}^{j-1}(-1)^{m}\frac{1}{m!(j+k-m-1)!}.
\end{equation}
Substituting $l=j+k-m-1$ we get
\begin{equation}\label{phi-2-1}
\varphi_{k,j}=(b-a)^{j+k-1}\sum_{l=k}^{j+k-1}(-1)^{j+k-l-1}\frac{1}{l!(j+k-l-1)!}.
\end{equation}
Now we multiply the matrix $T_1^{-1*}V$ from the right by $Q$. This means that odd columns are multiplied by $-1$.
To finish the proof of the self-adjointness of $T_1^{-1*}VQ$, one must show
that $\varphi_{j,k}-(-1)^{j+k}\varphi_{k,j}=0.$ We have
\begin{align}
&\frac{\varphi_{j,k}-(-1)^{j+k}\varphi_{k,j}}{(b-a)^{j+k-1}}=\sum_{l=0}^{k-1}(-1)^{l}\frac{1}{l!(j+k-l-1)!}-
\sum_{l=k}^{j+k-1}(-1)^{j+k-l-1}\frac{1}{l!(j+k-l-1)!}
\nonumber\\
&=\sum_{l=0}^{j+k-1}\frac{(-1)^l}{l!(j+k-l-1)!}=\frac{1}{(j+k-1)!}\sum_{l=0}^{j+k-1}(-1)^l\begin{pmatrix}
j+k-1 \\
l
\end{pmatrix}=0.\label{final}
\end{align}
The equality \eqref{eq1} is proved.

The equality \eqref{eq2} is implied by both~\eqref{eq1} and the following relations:
\begin{equation}\label{eq2_final}
QT_1T_2^{-1}S=QT_1V^{-1}, \quad VT_1^{-1}Q=Q\left(T_1^{-1*}VQ\right)^*Q.
\end{equation}

Now let us prove the equality \eqref{eq3}. Passing to inverse matrices in \eqref{eq3} and
taking into account the relations $V=ST_2, \quad V=V^*, \quad ST_1^{-1}=T_1^{-1*}S$ we obtain
\begin{equation}\label{eq3-inv}
VQ=\left(T_1^{-1*}VT_1^{-1}Q\right)^*=QT_1^{-1*}VT_1^{-1}.
\end{equation}
Multiplying the second equality in \eqref{eq3-inv} from the right by $T_1^{-1}$ we get
\begin{equation}\label{eq3-inv-2}
VQT_1^{-1}=QT_1^{-1*}V.
\end{equation}
Now let us prove \eqref{eq3-inv-2}. The entry $VQ$ has the form $\frac{(-1)^{k}}{(j+k-1)!}(b-a)^{j+k-1}.$
Therefore, the entry $VQT_1^{-1}$ equals
\begin{equation}\label{psy}
\psi_{j,k}=(b-a)^{j+k-1}\sum_{m=0}^{j-1}\frac{(-1)^{j+m}}{m!(j+k-m-1)!}=(b-a)^{j+k-1}(-1)^{k}\sum_{l=k}^{j+k-1}\frac{(-1)^{l+1}}{l!(j+k-l-1)!}.
\end{equation}
To calculate entries of $QT_1^{-1*}V$, we must multiply the matrix $T_1^{-1*}V$
from the left by $Q$. This means that odd rows are multiplied by $-1.$ Then, in accordance with~\eqref{phi}, the entry of the matrix $QT_1^{-1*}V$ is
\begin{equation}\label{final-1}
\mu_{j,k}=(b-a)^{j+k-1}(-1)^{k}\sum_{l=0}^{k-1}\frac{(-1)^{l}}{l!(j+k-l-1)!}.
\end{equation}
It is easily seen that $\psi_{j,k}=\mu_{j,k}$ (similarly to \eqref{final}). Equalities \eqref{eq3-inv-2} and \eqref{eq3} are established, and Theorem is completely proved
directly.
\end{proof}

\begin{proposition}\label{cor_M(0)}
Let $\Pi=\left\{\mathcal{H}, \Gamma_0, \Gamma_1\right\}$ be the boundary triplet for $A^*$
defined by \eqref{bt}, and let $M(\cdot)$ be the corresponding Weyl function. Then $B_K=M(0)=B_K^*.$
\end{proposition}
\begin{proof}
Combining Proposition \ref{A_0} (ii) with Theorem \ref{A_K} (ii) we arrive at the desired result.
\end{proof}

In the following theorem we describe all non-negative extensions of the operator $A$ as well as extensions having exactly $\kappa$ negative squares.


\begin{theorem}\label{th.2}
Let $\Pi=\left\{\mathcal{H}, \Gamma_0, \Gamma_1\right\}$ be the boundary triplet for operator $A^*$ defined by~\eqref{bt}, and let $B_K$ be the matrix defined by~\eqref{B_K}. Let also matrices $C,D\in\mathbb{C}^{2n\times 2n}$ satisfy the conditions $CD^*=DC^*, \det(CC^*+DD^*)\neq 0,$ and
\begin{equation}\label{A_CD-3}
A_{C,D}=A^*\upharpoonright{\ker}\left(D\Gamma_1-C\Gamma_0\right)=A_{C,D}^*.
\end{equation}
Then:
\item[(i)]
The following equivalence holds:
\begin{equation}\label{kappa_minus}
\kappa_-(A_{C,D})=\kappa \quad \Longleftrightarrow\quad \kappa_-(CD^*-DB_KD^*)=\kappa.
\end{equation}
In particular, $A_{C,D}\geq 0 \Longleftrightarrow CD^*-DB_KD^*\geq 0.$
\item[(ii)]
The operator $A_{C,D}$ is positively definite if and only if
the same holds for the matrix \\$CD^*-DB_KD^*$.
\end{theorem}
\begin{proof}
Due to Proposition \ref{cor_M(0)}, one has $B_K=M(0).$ To complete the proof, it suffices to use  Proposition~\ref{A_0} (iv).
\end{proof}
%

\section{Examples}

To facilitate the reading, let us provide four examples for $a=0, b=1$ and $n\in\{1,2,3,4\}.$
\begin{example}\label{example_1}
Let $n=1$, i.e., $Ay=-y''.$ Then
\begin{equation}\label{T-1}
T=\begin{pmatrix}
1 & 0 \\
1 & 1
\end{pmatrix},
\end{equation}
 and the boundary conditions from \eqref{dom_A_K} take the form:
\begin{equation}\label{bk-1}
\begin{cases}
f'(1)=f'(0) \\
f(1)=f'(0)+f(0)
\end{cases}.
\end{equation}
It follows from \eqref{TCS} and \eqref{B_K} that
\begin{equation}\label{matrices-1}
T_1=(1), \quad T_2=(1), \quad Q=(-1), \quad S=(1),
\end{equation}
and
\begin{equation}\label{B_K-ex-1}
B_K=\begin{pmatrix}
-1 & 1 \\
1 & -1
\end{pmatrix}
\end{equation}
is symmetric as required.
\end{example}

\begin{example}\label{example-2}
Let $n=2$, i.e., $Ay=y^{(iv)}.$ Then
\begin{equation}\label{T-2}
T=\begin{pmatrix}
1 & 0 & 0 & 0 \\
1 & 1 & 0 & 0 \\
\frac{1}{2} & 1 & 1 & 0 \\
\frac{1}{6} & \frac{1}{2} & 1 & 1
\end{pmatrix},
\end{equation}
and the boundary conditions from \eqref{dom_A_K} take the form:
\begin{equation}\label{bk-2}
\begin{cases}
f'''(1)=f'''(0) \\
f''(1)=f'''(0)+f''(0) \\
f'(1)=\frac{1}{2}f'''(0)+f''(0)+f'(0) \\
f(1)=\frac{1}{6}f'''(0)+\frac{1}{2}f''(0)+f'(0)+f(0)
\end{cases}.
\end{equation}
It follows from \eqref{TCS} and \eqref{B_K} that
\begin{equation}\label{matrices-2}
T_1=\begin{pmatrix}
1 & 0 \\
1 & 1
\end{pmatrix}, \quad
T_2=\begin{pmatrix}
\frac{1}{2} & 1 \\
\frac{1}{6} & \frac{1}{2}
\end{pmatrix}, \quad
Q=\begin{pmatrix}
1 & 0 \\
0 & -1
\end{pmatrix}, \quad
S=\begin{pmatrix}
0 & 1 \\
1 & 0
\end{pmatrix},
\end{equation}
and
\begin{equation}\label{B_K-ex-2}
B_K=\begin{pmatrix}
-12 & -6 & 12 & -6 \\
-6 & -4 & 6 & -2 \\
12 & 6 & -12 & 6 \\
-6 & -2 & 6 & -4
\end{pmatrix}.
\end{equation}
\end{example}

\begin{example}\label{example-3}
Let $n=3$, i.e., $Ay=-y^{(vi)}.$ Then
\begin{equation}\label{T-3}
T=\begin{pmatrix}
1 & 0 & 0 & 0 & 0 & 0 \\
1 & 1 & 0 & 0 & 0 & 0 \\
\frac{1}{2} & 1 & 1 & 0 & 0 & 0\\
\frac{1}{6} & \frac{1}{2} & 1 & 1 & 0 & 0 \\
\frac{1}{24} & \frac{1}{6} & \frac{1}{2} & 1 & 1 & 0 \\
\frac{1}{120} & \frac{1}{24} & \frac{1}{6} & \frac{1}{2} & 1 & 1
\end{pmatrix},
\end{equation}
and the boundary conditions are the following:
\begin{equation}\label{bk-3}
\begin{cases}
f^{(v)}(1)=f^{(v)}(0) \\
f^{(iv)}(1)=f^{(v)}(0)+f^{(iv)}(0) \\
f'''(1)=\frac{1}{2}f^{(v)}(0)+f^{(iv)}(0)+f'''(0) \\
f''(1)=\frac{1}{6}f^{(v)}(0)+\frac{1}{2}f^{(iv)}(0)+f'''(0)+f''(0) \\
f'(1)=\frac{1}{24}f^{(v)}(0)+\frac{1}{6}f^{(iv)}(0)+\frac{1}{2}f'''(0)+f''(0)+f'(0) \\
f(1)=\frac{1}{120}f^{(v)}(0)+\frac{1}{24}f^{(iv)}(0)+\frac{1}{6}f'''(0)+\frac{1}{2}f''(0)+f'(0)+f(0)
\end{cases}.
\end{equation}
Both \eqref{TCS} and \eqref{B_K} imply that
\begin{equation}\label{matrices-3}
T_1=\begin{pmatrix}
1 & 0 & 0 \\
1 & 1 & 0 \\
\frac{1}{2} & 1 & 1
\end{pmatrix}, \quad
T_2=\begin{pmatrix}
\frac{1}{6} & \frac{1}{2} & 1 \\
\frac{1}{24} & \frac{1}{6} & \frac{1}{2} \\
\frac{1}{120} & \frac{1}{24} & \frac{1}{6}
\end{pmatrix}, \quad
Q=\begin{pmatrix}
-1 & 0 & 0 \\
0 & 1 & 0 \\
0 & 0 & -1
\end{pmatrix}, \quad
S=\begin{pmatrix}
0 & 0 & 1 \\
0 & 1 & 0 \\
1 & 0 & 0
\end{pmatrix},
\end{equation}
and
\begin{equation}\label{B_K-ex-3}
B_K=\begin{pmatrix}
-720 & -360 & -60 & 720 & -360 & 60 \\
-360 & -192 & -36 & 360 & -168 & 24 \\
-60 & -36 & -9 & 60 & -24 & 3 \\
720 & 360 & 60 & -720 & 360 & -60 \\
-360 & -168 & -24 & 360 & -192 & 36 \\
60 & 24 & 3 & -60 & 36 & -9
\end{pmatrix}.
\end{equation}
\end{example}

\begin{example}\label{example-4}
Let $n=4$, i.e., $Ay=y^{(viii)}.$ Then
\begin{equation}\label{T-4}
T=\begin{pmatrix}
1 & 0 & 0 & 0 & 0 & 0 & 0 & 0 \\
1 & 1 & 0 & 0 & 0 & 0 & 0 & 0 \\
\frac{1}{2} & 1 & 1 & 0 & 0 & 0 & 0 & 0\\
\frac{1}{6} & \frac{1}{2} & 1 & 1 & 0 & 0 & 0 & 0 \\
\frac{1}{24} & \frac{1}{6} & \frac{1}{2} & 1 & 1 & 0 & 0 & 0 \\
\frac{1}{120} & \frac{1}{24} & \frac{1}{6} & \frac{1}{2} & 1 & 1 & 0 & 0 \\
\frac{1}{720} & \frac{1}{120} & \frac{1}{24} & \frac{1}{6} & \frac{1}{2} & 1 & 1 & 0 \\
\frac{1}{5040} & \frac{1}{720} & \frac{1}{120} & \frac{1}{24} & \frac{1}{6} & \frac{1}{2} & 1 & 1
\end{pmatrix},
\end{equation}
and the boundary conditions are the following:
\begin{equation}\label{bk-4}
\begin{cases}
f^{(vii)}(1)=f^{(vii)}(0) \\
f^{(vi)}(1)=f^{(vii)}(0)+f^{(vi)}(0) \\
f^{(v)}(1)=\frac{1}{2}f^{(vii)}(0)+f^{(vi)}(0)+f^{(v)}(0) \\
f^{(iv)}(1)=\frac{1}{6}f^{(vii)}(0)+\frac{1}{2}f^{(vi)}(0)+f^{(v)}(0)+f^{(iv)}(0) \\
f'''(1)=\frac{1}{24}f^{(vii)}(0)+\frac{1}{6}f^{(vi)}(0)+\frac{1}{2}f^{(v)}(0)+f^{(iv)}(0)+f'''(0) \\
f''(1)=\frac{1}{120}f^{(vii)}(0)+\frac{1}{24}f^{(vi)}(0)+\frac{1}{6}f^{(v)}(0)+\frac{1}{2}f^{(iv)}(0)+f'''(0)+f''(0) \\
f'(1)=\frac{1}{720}f^{(vii)}(0)+\frac{1}{120}f^{(vi)}(0)+\frac{1}{24}f^{(v)}(0)+\frac{1}{6}f^{(iv)}(0)+\frac{1}{2}f'''(0)+f''(0)+f'(0) \\
f(1)=\frac{1}{5040}f^{(vii)}(0)+\frac{1}{720}f^{(vi)}(0)+\frac{1}{120}f^{(v)}(0)+\frac{1}{24}f^{(iv)}(0)+\frac{1}{6}f'''(0)+\frac{1}{2}f''(0)+f'(0)+f(0)
\end{cases}.
\end{equation}
Both \eqref{TCS} and \eqref{B_K} imply that
\begin{equation}\label{matrices-4}
\begin{gathered}
T_1=\begin{pmatrix}
1 & 0 & 0 & 0 \\
1 & 1 & 0 & 0 \\
\frac{1}{2} & 1 & 1 & 0 \\
\frac{1}{6} & \frac{1}{2} & 1 & 1
\end{pmatrix}, \quad
T_2=\begin{pmatrix}
\frac{1}{24} & \frac{1}{6} & \frac{1}{2} & 1 \\
\frac{1}{120} & \frac{1}{24} & \frac{1}{6} & \frac{1}{2} \\
\frac{1}{720} & \frac{1}{120} & \frac{1}{24} & \frac{1}{6} \\
\frac{1}{5040} & \frac{1}{720} & \frac{1}{120} & \frac{1}{24}
\end{pmatrix}, \\
Q=\begin{pmatrix}
1 & 0 & 0 & 0 \\
0 & -1 & 0 & 0 \\
0 & 0 & 1 & 0 \\
0 & 0 & 0 & -1
\end{pmatrix}, \quad
S=\begin{pmatrix}
0 & 0 & 0 & 1 \\
0 & 0 & 1 & 0 \\
0 & 1 & 0 & 0 \\
1 & 0 & 0 & 0
\end{pmatrix},
\end{gathered}
\end{equation}
and
\begin{equation}\label{B_K-ex-4}
B_K=\begin{pmatrix}
-100800 & -50400 & -10080 & -840 & 100800 & -50400 & 10080 & -840 \\
-50400 & -25920 & -5400 & -480 & 50400 & -24480 & 4680 & -360 \\
-10080 & -5400 & -1200 & -120 & 10080 & -4680 & 840 & -60 \\
-840 & -480 & -120 & -16 & 840 & -360 & 60 & -4 \\
100800 & 50400 & 10080 & 840 & -100800 & 50400 & -10080 & 840 \\
-50400 & -24480 & -4680 & -360 & 50400 & -25920 & 5400 & -480 \\
10080 & 4680 & 840 & 60 & -10080 & 5400 & -1200 & 120 \\
-840 & -360 & -60 & -4 & 840 & -480 & 120 & -16
\end{pmatrix}.
\end{equation}
\end{example}

The authors express their gratitude to Prof. M. Malamud for posing the problem and permanent attention to the work, and to A. Ananieva and F. Gesztesy for useful discussions.

{\bf Yaroslav Igorovych Granovskyi}

Institute of Applied Mathematics and Mechanics, NAS of Ukraine

E-Mail: {\tt yarvodoley@mail.ru}

{\bf Leonid Leonidovych Oridoroga}

Institute of Applied Mathematics and Mechanics, NAS of Ukraine

\end{document}